\documentclass[smallextended,envcountsect]{svjour3}
\usepackage{latexsym,amsmath,amsfonts,amscd,amssymb,verbatim,array,mathdots}
\smartqed

\usepackage{chngcntr}
\counterwithin{equation}{section}

\DeclareMathOperator{\inte}{int}
\DeclareMathOperator{\gph}{gph}

\DeclareMathOperator{\Sol}{Sol}
\DeclareMathOperator{\OP}{OP}

\DeclareMathOperator{\Bo}{\mathbb{B}}

\DeclareMathOperator{\D}{\mathcal{D}}

\DeclareMathOperator{\Ri}{\mathcal{R}}
\DeclareMathOperator{\Oo}{\mathcal{O}}

\DeclareMathOperator{\Sa}{\mathcal{S}}

\DeclareMathOperator{\R}{\mathbb{R}}

\DeclareMathOperator{\N}{\mathbb{N}}
\DeclareMathOperator{\K}{\mathcal{K}}
\DeclareMathOperator{\Rc}{\mathcal{R}}

\usepackage{chngcntr}
\counterwithin{equation}{section}

\title{On the Nonemptiness and Boundedness of Solution Sets of Weakly Homogeneous Optimization Problems}

\author{Vu Trung Hieu}

\institute{Vu Trung Hieu \at LIP6, Sorbonne University, Paris, France\\
    E-mail: trung-hieu.vu@lip6.fr}

\date{Received: date / Accepted: date}

\begin{document}

    \maketitle

    % REQUIRED
    \begin{abstract} In this paper, we introduce a new class of optimization problems
whose objective functions are weakly homogeneous relative to the constraint
        sets. By using the normalization argument in asymptotic analysis, we prove two criteria for the nonemptiness and boundedness of the solution set of a weakly homogeneous optimization problem. Moreover, the normalization argument enables us to study the existence and
        stability of the solution sets of linearly parametric optimization problems. Several examples are given to illustrate the derived results.
    \end{abstract}

    % REQUIRED

    % REQUIRED

    \section{Introduction} The class of weakly homogeneous functions, which contains the subclass of all
    polynomial functions, has been introduced in the setting
    of variational inequality problems by Gowda and Sossa in \cite{GoSo18}, and studied recently by Ma, Zheng, and Huang in \cite{MZH20}. This setting is a simpler case
    of the one used initially by Gowda and Pang \cite{GP92}, and later by Flores-Baz\'{a}n \cite{FB2010}.

    In the present paper, we introduce a new class of weakly
    homogeneous functions (see Definition~\ref{def:wH} in the next section), which is stronger in some sense than that has been introduced in
    \cite{GoSo18}, and investigate the existence as well as the stability of the solution sets of
    \textit{weakly homogeneous optimization problems}. Asymptotic analysis plays an important
    role  and the normalization argument (see, e.g., \cite{AT2003,Cottle,LTY05}) is used
    effectively for our study.

For a given weakly homogeneous optimization problem, its \textit{kernel} is defined by the solution set of the corresponding asymptotic optimization problem. The kernel is always a closed cone if it is
    nonempty. Based on the kernel notion, we establish two criteria for the nonemptiness and boundedness of the
    solution sets of weakly homogeneous optimization problems. The first one says that if the kernel is the trivial cone then the solution set is nonempty and bounded. The second one concerns the pseudoconvexity of the objective
    function when the kernel is non-trivial and the constraint set is convex. These results are considered as extensions of the Frank-Wolfe
    type theorem and the Eaves type theorem for polynomial optimization problems in
    \cite{Hieu2018}.

    The kernel and the domain of an affine variational inequality were introduced in \cite{FaPa03,GP}. They are useful in the study of the existence and stability of solution sets of parametric affine variational inequality problems. For weakly homogeneous optimization problems, these notions can be developed to study the existence and stability of solutions to \textit{linearly parametric optimization problems}.

    The organization of the paper is as follows.
    Section \ref{sec:pre} gives a brief introduction to
    asymptotic cones, weakly homogeneous functions, pseudoconvexity, and
    optimization problems.  Asymptotic problems are discussed in Section \ref{sec:asym} while two results on the nonemptiness and boundedness of solution sets are shown
    in Section \ref{sec:comp}. Sections \ref{sec:linear} and \ref{sec:usc} are devoted to
    investigate the existence and stability of solutions to linearly parametric
    optimization problems, respectively.

    \section{Preliminaries}\label{sec:pre}

    Let $S$ be a nonempty closed set in $\R^n$, its \textit{asymptotic cone} $S^{\infty}$ is defined \cite{AT2003}  by
    $$
    S^{\infty}=\Big\{ v\in\R^n:\exists t_k\to +\infty, \exists x^k\in S \text{ with
    } \lim_{k\to\infty} \frac{x^k}{t_k}=v\Big\}.
    $$
    It is well-known that $S^{\infty}$ is a closed cone and that $S^{\infty}$ is the
    trivial
    cone if and only if $S$ is bounded. Besides, if $S$ is convex then
    $S^{\infty}$ coincides with the recession cone of $S$ containing
    vectors $v \in\R^n$ such that $x+tv\in S$ for any $x\in S$ and $t\geq 0$, i.e.,
    $S=S+S^{\infty}$.

    For  an arbitrary cone $C$ in $\R^n$, we recall that the \textit{dual cone} $C^*$ of $C$ is given by $$C^*=\{u\in \R^n:\left\langle u,v \right\rangle \geq 0, \forall v\in C\}$$ and that
    $u$ belongs to the interior of $C^*$, denoted by $\inte(C^*)$, if and
    only if
    $\left\langle u,v \right\rangle>0$ for all $v\in C$ and $v\neq 0$.

    \medskip
    Assume that $C$ is a closed cone in $\R^n$, $K$ is an unbounded and closed subset
    of $C$, and $f:C\to\R$ is a continuous function on $C$. Clearly, $K^{\infty}$ is a
    subset of $C$.
    \begin{definition}\label{def:wH} The function $f$ is said to be \textit{weakly homogeneous
            of degree $\alpha$ relative to $K$} if there exists a positively homogeneous
        continuous function $h$ of degree $\alpha>0$ on $C$, i.e.
        $h(tx)=t^{\alpha}h(x)$ for all $x\in C$ and $t>0$, such that
        $f(x)-h(x)=o(\|x\|^{\alpha})$ on $K$. Then, one says that $h$ is an
        \textit{asymptotically homogeneous function} of $f$ relative to $K$.
    \end{definition}

   Assume that $f$ is weakly homogeneous
    of degree $\alpha$ relative to $K$.  The asymptotically homogeneous function $h$ in Definition \ref{def:wH} may be not
    unique (see Example \ref{ex:1}). We denote by $[f_{\alpha,K}^{\infty}]$ the
    set of asymptotically homogeneous functions of $f$ relative to
    $K$, and by $\Oo_K^{\alpha}$ the space of all continuous
    functions $p$ on $C$, such that $p(x)=o(\|x\|^{\alpha})$ on $K$. Clearly, if
    $p\in \Oo_K^{\alpha}$, then $f+p$ is
    also weakly
    homogeneous of degree $\alpha$ relative to $K$.

    \begin{remark} The notion of weakly asymptotic function in Definition
        \ref{def:wH} is
        different from the notion in the monograph by Auslender and
        Teboulle~\cite[Definition 2.5.1]{AT2003} as well as that of Gowda and Sossa
        \cite{GoSo18}. In some sense, the notion in Definition \ref{def:wH} is
        stronger  (see Example \ref{ex:1} and Remark \ref{rmk:strong}) than that of
        \cite{GoSo18} in which it is required that $f(x)-h(x)=o(\|x\|^{\alpha})$ on the whole
        cone $C$ instead of only on the subset $K$. Hence, in Definition \ref{def:wH} we
        emphasize that $f$ is weakly homogeneous (of degree $\alpha$) \textit{relative
            to $K$}.
    \end{remark}

    \begin{example}\label{ex:1}
        Consider the cone $C:=\R_+^2:=\{(x_1,x_2)\in\R^2:x_1\geq 0,x_2\geq 0 \}$, the set
        $K:=\{(x_1,x_2)\in\R^2:(x_1-2)^2+(x_2-2)^2\leq 1
        \}\cup\{(x_1,x_2)\in\R^2:x_1\geq 1,x_2=0\},$ and the function
        $f(x_1,x_2):=x_1x_2+\sqrt{x_1}$. It is clear that
        $K^{\infty}=\{(x_1,x_2)\in\R^2:x_1\geq 0,x_2=0\}$ and $h(x_1,x_2)=\sqrt{x_1}$
        is an asymptotically homogeneous function of $f$ relative to $K$, i.e., $f$ is
        weakly
        homogeneous of degree $\alpha=1/2$ relative to $K$. We can see that $\bar
        h(x_1,x_2)=\sqrt{x_1}+\sqrt{x_2}$ is also a different asymptotically homogeneous
        function of $f$ relative to $K$; Clearly, $h$ is different from $\bar h$ on
        $\R_+^2$ but they have the same values on $K^{\infty}$.
        Meanwhile, the function $f$ is a weakly homogeneous function of degree $\alpha=2$ in the
        sense of \cite{GoSo18}, in which the involved asymptotic function is $\hat
        h(x_1,x_2)=x_1x_2$.
    \end{example}

    \begin{example}\label{ex:3}
        Let $K$ be an unbounded and closed subset of $\R^n$ and $f$ be a polynomial of degree $d$ in $n$ variables. Then, $f$ is weakly homogeneous of degree $\alpha$ relative to $K$, where $\alpha$ is a positive integer, $\alpha\leq d$.  Indeed, assume that $f=\sum_{\beta\in \mathcal{B}}a_{\beta}x^{\beta}$, where $\mathcal{B}$ is a finite subset of $\N^n$, $\beta=(\beta_1,\dots,\beta_n)$, $x^{\beta}=x_1^{\beta_1}\cdots x_n^{\beta_n}$, $|\beta|=\beta_1+\dots+\beta_n\leq d$, $a_{\beta}\in \R$. Let $I:=\{i: \exists x\in K^{\infty}, x_i\neq 0\}$ and $\mathcal{B}_I:=\{(\beta_1,\dots,\beta_n)\in \mathcal{B}: \beta_j=0 \text{ or } j\in I, \forall j=1,\dots,n\}$. We can check that the following homogeneous polynomial is an asymptotically homogeneous function of $f$ relative to $K$:
        $$h= \begin{cases}
            \sum_{\beta\in \mathcal{B}_I, |\beta|=\alpha }a_{\beta}x^{\beta} &\mbox {if } \ \mathcal{B}_I \neq \emptyset, \\
            0 & \mbox{if } \ \mathcal{B}_I = \emptyset,
        \end{cases}$$
    \end{example}
    where $\alpha: =\max\{|\beta|:\beta\in \mathcal{B}_I\}$.

    Let $K$ be a nonempty closed subset in $\R^n$
    and $f:K\to\R$ be a continuous function. The minimization problem with the constraint set $K$ and the objective function $f$ is written formally as follows:
    \begin{equation*}\label{OP} \OP(K,f): \ \ {\rm minimize}\ \; f(x)\ \ {\rm
            subject \ to}\ \  x\in K. \quad \quad
    \end{equation*}
    The solution set of
    $\OP(K,f)$ is denoted by $\Sol(K,f)$. By the closedness of $K$ and the
    continuity of $f$, $\Sol(K,f)$ is closed. Clearly, if $\Sol(K,f)$ is nonempty
    then $f$ is bounded from below on $K$. If $K$ is bounded, then $\Sol(K,f)$ is nonempty by the Bolzano-Weierstrass theorem.

    \begin{remark}\label{rmk2} Suppose that $K$ is a cone and $f$ is positively
        homogeneous function of degree $\alpha>0$ relative to $K$. We can see that if
        $\Sol(K,f)$ is nonempty then this set is a closed cone. Besides, it is not difficult to check that the following conditions are equivalent: $\Sol(K,f)$ is nonempty; $f$ is
        bounded from below on $K$; $0$ is a solution of $\OP(K,f)$; $f$ is
        non-negative on $K$. Furthermore, if one of these conditions is satisfied, then $\Sol(K,f)$ is the set of zero points of $f$ in $K$,
        i.e. $\Sol(K,f)=\{x\in K: f(x)=0\}$.
    \end{remark}

    Assume that $f$ is differentiable on an open subset $U\subset \R^n$. The gradient of $f$ is denoted by $\nabla f$.
    The function $f$ is said to be \textit{pseudoconvex} on $U$ if, for any $x,y\in U$ such that $\langle \nabla f(x),y-x\rangle\geq 0$, we have $f(y)\geq f(x)$. Recall that $f$ is pseudoconvex on $U$ if and only if $\nabla f$ is \textit{pseudomonotone} on $U$, i.e., if for any $x,y\in U$ such that $\langle \nabla f(x),y-x\rangle\geq 0$ then $\langle \nabla f(y),y-x\rangle\geq 0$ (see, e.g., \cite[Theorem 4.4]{ALM2014}). If $f$ is convex on $U$ then $f$ is pseudoconvex on $U$. Note that if $f$, $g$ are two convex functions on a convex set $K$ then the sum $f+g$ is convex on $K$.

    \begin{remark}\label{rmk:Minty} Assume that $K$ is convex and $f$ is pseudoconvex on an open set $U$ containing $K$. If $x^0\in \Sol(K,f)$ then $\langle \nabla f(x),x-x^0\rangle\geq 0$ for all $x\in K$. Indeed, since $x^0$ is a solution of $\OP(K,f)$, one has $\langle \nabla f(x^0),x-x^0\rangle\geq 0$ for all $x\in K$ (see, e.g., \cite[Proposition 5.2]{ALM2014}). The pseudomonotonicity of the gradient implies that $\langle \nabla f(x),x-x^0\rangle\geq 0$ for all $x\in K$. Conversely, if the point $x^0\in K$ satisfied $\langle \nabla f(x^0),x-x^0\rangle\geq 0$ for all $x\in K$ then $x^0\in \Sol(K,f)$ (see, e.g., \cite[Proposition 5.3]{ALM2014}).
    \end{remark}

Let $V$ be a closed cone contained in $K^{\infty}$. The following lemma provides a necessary condition for a vector to belong to the interior of $V^*-\nabla f(K)$. This lemma will be used in the proof of
    Theorems \ref{thm:4} and \ref{thm:usc_Kinfty}.

    \begin{lemma}\label{lm:conc}Assume that $f$ is differentiable on an open set $U$ containing  $K$. If $u\in\inte(V^*-\nabla f(K))\neq\emptyset$, then
        for each $v\in V\setminus\{0\}$ there is $x\in K$ such that
        $\left\langle \nabla f(x)+u,v \right\rangle> 0$.
    \end{lemma}

    \begin{proof}
Assume that $\inte(V^*-\nabla f(K))$ is nonempty and $u$ belongs to this set. There exists an $\varepsilon>0$ such that
        $\Bo(u,\varepsilon)\subset V^*-\nabla f(K)$,
        where $\Bo(u,\varepsilon)$ is the open ball of radius $\varepsilon$
        centered at $u$. Assume on the contrary that there is $\bar v\in
V\setminus\{0\}$  such that
        $\left\langle \nabla f(x)+u,\bar v \right\rangle\leq 0$ for all $x\in K$.
        Let $a$ be an arbitrary point in $\Bo(u,\varepsilon)\subset \inte(V^*-\nabla f(K))$. Then, there exist $w\in
V^*$ and $\bar x\in K$  such that $a=w-\nabla f(\bar x)$. This leads to
        $$ \left\langle a-u,\bar v \right\rangle=\left\langle w,\bar v \right\rangle-\left\langle \nabla f(\bar x)+u,\bar v \right\rangle\geq 0.$$
        As the latter holds for any $a\in \Bo(u,\varepsilon)$, it follows that $\bar v=0$; hence one gets a contradiction. \qed
    \end{proof}

    The following assumptions will be essential throughout the paper: \textit{$K$ is nonempty, closed, and unbounded; $f$ is continuous on the cone $C$ containing $K$;  $f$ is weakly homogeneous of degree $\alpha>0$ relative to $K$.}

    \section{Properties of the Asymptotic Optimization Problem}\label{sec:asym}
  This section shows some basic properties of the
    asymptotic optimization problem associated to $\OP(K,f)$. The
    asymptotic optimization problem plays a vital role to investigate the behavior of the original problem at infinity.

    The following proposition says that asymptotically homogeneous
    functions in $[f_{\alpha,K}^{\infty}]$ have the same values on the asymptotic cone $K^{\infty}$; hence the asymptotic optimization problem is defined uniquely.

    \begin{proposition}\label{prop:hh} Suppose that $h,\bar h$ are in $[f_{\alpha,K}^{\infty}]$. Then, $h=\bar h$ on $K^{\infty}$; hence	$\Sol(K^{\infty},h)=\Sol(K^{\infty},\bar h).$
    \end{proposition}
    \begin{proof} Let $h,\bar h$ be two asymptotically homogeneous functions of degree
        $\alpha$ of $f$ relative to $K$. Assume that $f=h+g$ and $ f=\bar h+\bar g$,
        where $g, \bar g \in\Oo_K^{\alpha}$. Let $v\in K^{\infty}$ be given. There exist two sequences $\{t_k\}\subset
        \R_+$ and
        $\{x^k\}\subset K$ such that $t_k\to +\infty$ and $\lim_{k\to\infty}
        t_k^{-1}{x^k}=v$. Taking $y_k=t_kv$, one has $y_k\in K^{\infty}$ and
        \begin{equation}\label{hh}
            h(t_kv)+g(t_kv)= f(y_k)=\bar h(t_kv)+\bar g(t_kv),
        \end{equation}
        for any $k$. By dividing the equations in \eqref{hh} by $t_k^{\alpha}$ and letting
        $k\to+\infty$, we get $h(v)=\bar h(v)$. This conclusion holds for any $v\in
        K^{\infty}$; hence $h=\bar h$ on $K^{\infty}$. \qed
    \end{proof}

    According
    to Proposition \ref{prop:hh}, $\Sol(K^{\infty},h)$ does not depend on the
    choice of
    $h$ in $[f_{\alpha,K}^{\infty}]$. We can write a member of
    $[f_{\alpha,K}^{\infty}]$ simply by $f^{\infty}$ when no confusion arises.
    We
    denote  $$\K(K,f):=\Sol(K^{\infty},f^{\infty}).$$
    From Remark \ref{rmk2}, one can see that if $\K(K,f)$ is nonempty
    then $\K(K,f)$ is the set of zero points of $f^{\infty}$ in $K^{\infty}$, i.e.,
    $\K(K,f)=\{x\in K^{\infty}: f^{\infty}(x)=0\}$. So, we call $\K(K,f)$
    the \textit{kernel} of the weakly homogeneous optimization problem $\OP(K,f)$.

    \begin{proposition}\label{prop:union}
        Assume that $K$ is convex. One has the following inclusion
        \begin{equation}\label{cup}
            \bigcup_{p\in \Oo_K^{\alpha}}\Sol(K,f+p)^{\infty}\subset \K(K,f).
        \end{equation}
        Furthermore, if $K$ is a cone then the reverse inclusion of \eqref{cup} holds.
    \end{proposition}

    \begin{proof} Assume $f=f^{\infty}+g$, where $g \in\Oo_K^{\alpha}$.  Let
        $p\in\Oo_K^{\alpha}$ be given. Suppose that $\bar x\in\Sol(K,f+p)^{\infty}$. There exist two sequences $\{x^k\}\subset\Sol(K,f+p)$ and
        $\{t_k\}\subset \R_+$ such that $t_k \to+\infty$ and $t_k^{-1}{x^k}\to\bar x$. By assumptions, for each $x^k$, one has
        \begin{equation}\label{xk}
            f(y)+p(y)\geq f(x^k)+p(x^k)
        \end{equation}
        for all $y\in K$.
        Let $u\in K$ be fixed. Since $K$ is convex, $K^{\infty}$ coincides with the recession cone of $K$. For every $v\in K^{\infty}$, one
        has $u+t_kv\in K$ for any $k$. Now, \eqref{xk} implies that
        \begin{equation}\label{xk2}
            f^{\infty}(u+t_kv)+ g(u+t_kv)+p(u+t_kv)\geq f^{\infty}(x^k)+ g(x^k)+p(x^k),
        \end{equation}
        for any $k$. As $g,p\in\Oo_K^{\alpha}$, one can see that the following
        values:
        $g(u+t_kv)/t_k^{\alpha}$,  $g(x^k)/t_k^{\alpha}$, $p(u+t_kv)/t_k^{\alpha}$,
        and
        $p(x^k)/t_k^{\alpha}$ go to $0$ together as $k\to+\infty$. Therefore, by dividing the
        inequality in
        \eqref{xk2} by
        $t_k^{\alpha}$  and letting $k\to+\infty$, we obtain $f^{\infty}(v)\geq
        f^{\infty}(\bar x)$.
        The assertion holds for every $v\in K^{\infty}$. We conclude that $\bar x\in
        \K(K,f)$, hence the inclusion \eqref{cup} is proved.

        Assume that $K$ is a cone. Clearly,
        $f^{\infty}=f+(f^{\infty}-f)$ and
        $f^{\infty}-f\in \Oo_K^{\alpha}$. Hence, one has
        $\K(K,f)=\Sol(K,f^{\infty})=\Sol(K,f+(f^{\infty}-f))$. It follows
        that the reverse inclusion holds. \qed
    \end{proof}

    \begin{proposition}\label{prop:nonempty} If $f$ is bounded from below on $K$, then $\K(K,f)$ is nonempty.
    \end{proposition}
    \begin{proof} Let $v\in K^{\infty}$ be arbitrary. There exist two sequences $\{t_k\}\subset
        \R_+$ and
        $\{x^k\}\subset K$ such that $t_k\to +\infty$ and $\lim_{k\to\infty}
        t_k^{-1}{x^k}=v$. Assume that $\gamma$ is a bound from below of $f$ on $K$. One has $f(x^k)\geq \gamma$, for any $k$. By dividing the inequality by $t_k^{\alpha}$ and letting
        $k\to+\infty$, we get $f^{\infty}(v)\geq 0$. This conclusion implies $f^{\infty}$ is non-negative on $K^{\infty}$. According to Remark \ref{rmk2}, $\Sol(K^{\infty},f^{\infty})$ is nonempty. \qed
    \end{proof}

    \section{Nonemptiness and Boundedness of Solution Sets}\label{sec:comp}

In this section, we introduce two criteria for the nonemptiness and boundedness of $\Sol(K,f)$ that are considered as extended versions of the Frank-Wolfe type theorem and the Eaves type theorem for polynomial optimization problems in \cite{Hieu2018}. Their proofs are modified from \cite[Theorem 3.1]{Hieu2018} and \cite[Theorem 3.2]{Hieu2018}.

    The following theorem asserts that the nonemptiness and boundedness property holds for $\Sol(K,f)$ provided that the kernel is the trivial cone.

    \begin{theorem}\label{thm:trivial} If $\K(K,f)=\{0\}$, then $\Sol(K,f)$ is
        nonempty and bounded.
    \end{theorem}

    \begin{proof} Suppose that $\K(K,f)=\{0\}$, i.e., $f^{\infty}(v)> 0$ for all $v\in K^{\infty}$ and $v\neq 0$. Take $x^0\in K$ and define
        $M=\{x\in K: f(x^0)\geq f(x)\}$.
        Clearly, $M$ is nonempty and closed.  Besides, it is not difficult to check that
        $\Sol(K,f)=\Sol(M,f)$.

        We claim that $M$ is bounded. Indeed, on the contrary,
        we suppose that $M$ is unbounded. Then there exists a sequence $\{x^k\}\subset
        M$ such that  $\|x^k\|\to +\infty$. Assume, without loss of generality, that $\|x^k\|^{-1}x^k\to v$ and $ v\in K^{\infty}\setminus\{0\}$. For each $k$,
        we have $f(x^0)\geq f(x^k)$. Dividing both sides of this inequality by
        $\|x^k\|^{\alpha}$	and letting $k\to +\infty$, we obtain $0\geq
        f^{\infty}(v)$. This is impossible because of
        $v\neq 0$ and $f^{\infty}(v)> 0$.  Hence, $M$ is bounded.

        Since $M$ is compact, the desired result follows from the
        Bolzano-Weierstrass theorem.	\qed
    \end{proof}

    The next example illustrates Theorem \ref{thm:trivial}.

    \begin{example}
        Consider the weakly homogeneous optimization problem given by $f$ and  $K$ as in  Example~\ref{ex:1}. One can see that $\K(K,f)=\{0\}$. According to Theorem \ref{thm:trivial}, the solution set is nonempty and bounded. Meanwhile, we can see that $\Sol(K,f)=\{(1,0)\}$.
    \end{example}

    \begin{remark}\label{rmk:strong} Our new notation (weakly homogeneous function) is stronger than that of \cite{GoSo18} in sense that our
        kernel is a subset of $\Sol(K^{\infty},\hat h)$, where $\hat h$ is the
        asymptotically homogeneous function of $f$ defined as in \cite{GoSo18}. In Example~\ref{ex:1},
        the asymptotic function in the sense of \cite{GoSo18} is $\hat
        h(x_1,x_2)=x_1x_2$; hence, one has $\Sol(K^{\infty},\hat h)=\{(x_1,0):x_1\geq
        0\}\supset \{0\} = \K(K,f)$.
    \end{remark}

 The following theorem considers the case that $\K(K,f)$ is non-trivial (i.e., unbounded) and the objective function $f$ is pseudoconvex. A relationship between $\K(K,f)$ and $\nabla f(K)$ enables us to see
    the nonemptiness and boundedness of $\Sol(K,f)$.

    \begin{theorem}\label{pseudo} Assume that $K$ is convex and $\K(K,f)$ is non-trivial. If $f$ is pseudoconvex on an open set $U$ containing $K$, then the following statements are equivalent:
        \begin{description}
            \item[\rm(a)] For each $v\in\K(K,f)\setminus\{0\}$, there exists $x\in K$ such that $\left\langle \nabla f(x), v \right\rangle > 0$;
            \item[\rm(b)] $\Sol(K,f)$ is nonempty and bounded.
        \end{description}
    \end{theorem}

    \begin{proof} Suppose that $f$ is pseudoconvex on an open set
        $U$ containing $K$.

        $\rm(a) \Rightarrow \rm(b)$ Suppose that $\rm(a)$ holds.  For each
        $k=1,2,\dots$, we denote $K_k=K\cap \overline\Bo(0,k)$, where
        $\overline\Bo(0,k)$ is the closed ball of radius $k$
        centered at the origin. Clearly, $K_k$ is compact and convex. We can assume that $K_k$ is
        nonempty. From the Bolzano-Weierstrass theorem, $\OP(K_k,f)$ has a solution, say $x^k$.

        We claim that the sequence $\{x^k\}$ is bounded. On the contrary, suppose that this sequence is unbounded and that $x^k\neq 0$ for all $k$, 	$\|x^k\|^{-1}x^k\to v$, where $ v \in K^{\infty}$ with $\| v\|=1$.
        For each $k$, we have
        \begin{equation}\label{f_Kk}
            f(x) \geq f(x^k), \ \forall x\in K_k.
        \end{equation}
        By fixing $x\in K_1$, hence $x\in K_{k}$ for any $k$,  dividing both sides of the inequality in \eqref{f_Kk} by $\|x^k\|^{\alpha}$, and letting $k\to+\infty$, we get
        $0\geq f^{\infty}(v)$. This leads to $ v\in\K(K,f)\setminus\{0\}$. For each
        $k$, since $x^k\in\Sol(K_k,f)$ and $f$ is pseudoconvex on the convex set $K_k$, from Remark
        \ref{rmk:Minty}, we have
        \begin{equation}\label{nab_f}
            \langle \nabla f(x),x-x^k\rangle\geq 0, \ \forall x\in K_k.
        \end{equation}
        Take an arbitrary $x\in K$. For $k$ large enough, $x$ belongs to $K_k$.
        Dividing both
        sides of the inequality in \eqref{nab_f} by $\|x^k\|$	and letting $k\to +\infty$, we obtain
        $0\geq \langle \nabla f(x),v\rangle$. This
        contradicts $\rm(a)$; hence $\{x^k\}$ is bounded.

        We can assume that  $x^k\to \bar x$. From \eqref{f_Kk}, by the continuity of $f$, we see that  $\bar x$ solves $\OP(K,f)$, so $\Sol(K,f)$ is nonempty.

        Let us give a brief proof for the boundedness of $\Sol(K,f)$. We suppose on the contrary that there is an unbounded solution sequence $\{x^k\}$, $\|x^k\|^{-1}x^k\to v$, where $ v \in K^{\infty}$ and $\| v\|=1$. The inequalities in \eqref{f_Kk} and \eqref{nab_f} hold for any $x\in K_k$. By repeating the previous arguments, we can get a conclusion which contradicts $\rm(a)$. Hence, the first assertion is proved.

        $\rm(b) \Rightarrow \rm(a)$ Since $K$ is convex, one has $K=K+K^{\infty}$.
        Suppose that $\Sol(K,f)$ is nonempty and bounded, but on the contrary there exists  $v\in\K(K,f)\setminus\{0\}$ such that $\langle \nabla f(x),v\rangle\leq 0$
        for all $x\in K$. Let $\bar x$ be a solution of $\OP(K,f)$. For any
        $t\geq 0$, one has $\bar x+tv\in K$ and
        $\langle \nabla f(\bar x+tv),v\rangle\leq 0$. This implies that
        $$\langle \nabla f(\bar x+tv),\bar x-(\bar x+tv)\rangle\geq 0.$$
        The pseudoconvexity of $f$ yields  $f(\bar x)\geq f(\bar x+tv)$. Thus, $\bar x+tv\in\Sol(K,f)$ for any $t\geq 0$. It follows that $\Sol(K,f)$ is unbounded which contradicts our assumption. Hence, $\rm (a)$ holds. \qed
    \end{proof}

    To illustrate Theorem \ref{pseudo}, we provide the following example.

    \begin{example}\label{ex:2}
        Consider the objective function $f(x_1,x_2):=\frac{1}{2}x_1^2-x_1x_2+x_2^{3}$ and the constraint set $K:=\{(x_1,x_2)\in\R^2: x_1x_2\geq 1, x_2\geq 2\}\subset \R^2_{+}$. Clearly,  $U:=\{(x_1,x_2)\in\R^2: x_1> 0, x_2 > 1\}$ is an open set containing $K$. The gradient and the Hessian matrix of $f$ on $U$ are  determined respectively by
        \begin{equation}\label{eq:examp}
            \nabla f= \begin{bmatrix}
                x_1-x_2 \\
                -x_1+3x_2^{2}
            \end{bmatrix}, \ \ H_f=\begin{bmatrix}
                1& -1 \\
                -1 \ & \ 6x_2
            \end{bmatrix}.
        \end{equation}
        It is easy to check that $K$ is convex and that $H_f$ is positive semidefinite on the open set
        $U$; hence $f$ is convex on $K$. One has $K^{\infty}=\R_{+}^2$ and
        $f^{\infty}(x_1,x_2)=x_2^{3}$ with $\alpha=3$. This yields
        \begin{equation}\label{exK}
            \K(K,f)=\{(x_1,x_2)\in\R^2: x_1\geq 0,x_2=0\}.
        \end{equation}
        For any $v$ in  $\K(K,f)\setminus\{0\}$, $v=(\beta,0)$, where $\beta > 0$, by choosing $\bar x=(3,2)\in K$, we get $\left\langle \nabla f(\bar x), v \right\rangle=\beta > 0$. According to Theorem~\ref{pseudo}, the solution set of $\OP(K,f)$ is nonempty and bounded.
    \end{example}

    \section{Linearly Parametric Optimization Problems}\label{sec:linear}

    In the sequel, we assume that
    $\alpha>1$. This section investigates the nonemptiness and boundedness of the solution sets of \textit{linearly parametric}  optimization problem
    $\OP(K,f_u)$, where $u\in\R^n$ and
    $$f_u(x):=f(x)+\left\langle u,x
    \right\rangle.$$

\begin{theorem}\label{thm:3}
        Assume that $f$ is bounded from below on $K$ and $\inte(\K(K,f)^*)$ is nonempty. If $u\in \inte(\K(K,f)^*)$, then $\Sol(K,f_u)$ is nonempty and bounded.
    \end{theorem}

    \begin{proof} Let $\gamma$ be a lower bound of $f$ on $K$. Suppose that
        $u$ is an element of $\inte(\K(K,f)^*)$. We now prove the nonemptiness of
        $\Sol(K,f_u)$.  For each $k=1,2,\dots$, we define $K_k=K\cap \overline\Bo(0,k)$. Clearly, $K_k$ is compact. We can
        assume that $K_k$ is nonempty. By the Bolzano-Weierstrass theorem,
        $\OP(K_k,f_u)$ has a solution, say $x^k$. One has
        \begin{equation}\label{op}
            f(y)+\left\langle u,y \right\rangle \geq f(x^k)+\left\langle u,x^k \right\rangle
        \end{equation}
        for all $y\in K_k$.

        One claims that $\{x^k\}$ is a bounded sequence. Indeed, on the contrary, suppose that the sequence is unbounded. We can assume that  $x^k\neq 0$ for all $k$,
        $\|x^k\|^{-1}x^k\to\bar x$ with $\bar x \in K^{\infty}$ and $\|\bar x\|=1$.
        Let $y$ in $K_1$ be fixed. By dividing the inequality \eqref{op} by $\|x^k\|^{\alpha}$ and letting $k\to+\infty$, since $\alpha>1$, we obtain $0\geq f^{\infty}(\bar x)$. By the boundedness from below of $f$ on $K$, one can see that $f^{\infty}(\bar x) \geq 0$. It follows that $f^{\infty}(\bar x)= 0$; hence  $\bar x \in \K(K,f)$. From \eqref{op}, we can see that
        $$f(y)+\left\langle u,y \right\rangle \geq \gamma+\left\langle u,x^k \right\rangle.$$
        This leads to $0 \geq \left\langle u,\bar x\right\rangle$ which contradicts our assumption $u\in\inte(\K(K,f)^*)$. Thus, the sequence
        $\{x^k\}$ is bounded.

        We can suppose that  $x^k\to z$. Let $y\in K$ be given, then $y\in K_k$ for $k$
        large enough and \eqref{op} holds. By the continuity of $f$ and \eqref{op}, we have $	f(y)+\left\langle u,y \right\rangle  \geq f(z)+\left\langle u,z
        \right\rangle$. This conclusion holds for all $y\in K$; hence $z$ solves
        $\OP(K,f_u)$, and the nonemptiness of $\Sol(K,f_u)$ follows.

To prove the boundedness of $\Sol(K,f_u)$, on the contrary, assume that there exists an unbounded sequence  $\{x^k\}\subset\Sol(K,f_u)$,  $x^k\neq 0$ for all $k$, and
        $\|x^k\|^{-1}x^k\to\bar x$.
        Repeating the previous arguments for $\{x^k\}$, we also obtain the facts that
        $\bar x \in \K(K,f)$ and $0\geq \left\langle u,\bar x \right\rangle$. This implies $u\notin \inte(\K(K,f)^*)$ which contradicts
        our assumption. %The proof is complete.
        \qed
    \end{proof}

We denote by $\Rc(K,f)$ the set of all $u\in \R^n$ such that $\OP(K,f_u)$ has a solution, i.e.,
$$\Rc(K,f):=\{u\in \R^n: \Sol(K,f_u) \neq \emptyset\}.$$
%If $K$ is convex and $f$ is pseudoconvex on $K$, then $-\nabla f(K)=\Rc(K,f)$ (see, e.g., \cite[Theorem 1.37]{ALM2014}).

    \begin{corollary}%\label{cor:ker2}
        Assume that $K$ is a pointed cone and $f$ is bounded from below on $K$. Then, $\Rc(K,f)$ is nonempty.
    \end{corollary}
    \begin{proof} By Theorem \ref{thm:3}, one has $\inte(\K(K,f)^*)\subset \Rc(K,f)$.
        Because $K$ is a cone, we have $\K(K,f)\subset K$ and $K^*\subset \K(K,f)^*$.
        Since the cone $K$ is pointed, $K^*$ has a nonempty interior; hence $\inte(\K(K,f)^*)$ is nonempty. This implies the nonemptiness of $\Rc(K,f)$. \qed
    \end{proof}

    \begin{example}\label{ex:4}
        Let us consider the polynomial optimization problem given by
        $f:=x_2(x_1x_2-1)$ and $K:=\{(x_1,x_2)\in \R^2: x_1x_2\geq 1,  x_1 -x_2 \geq 0,x_2\geq 0 \}$.  It is clear that $f$ is bounded from below on $K$. One has $f^{\infty}(x_1,x_2)=x_1x_2^{2}$ with $\alpha=3$ and $K^{\infty}=\{(x_1,x_2)\in \R^2:x_1 -x_2 \geq 0,x_2\geq 0 \}$. Therefore,
        $\K(K,f)=\{(x_1,x_2)\in\R^2: x_1\geq 0,x_2=0\}$. The interior of the dual cone of the kernel is determined by
        $$\inte(\K(K,f)^*)=\{(u_1,u_2)\in\R^2: u_1 > 0\}.$$
        According to Theorem \ref{thm:3}, the solution set of $\OP(K,f_u)$ is nonempty and bounded for any $u=(u_1,u_2)$ with $u_1> 0$. To see the (un)boundedness of $\OP(K,f_u)$ when $u\notin \inte(\K(K,f)^*)$, we consider the two following cases where $u$ belongs to the boundary of $\K(K,f)^*$ as follows:
        \begin{itemize}
            \item First, we take $u=(0,1)$, then $f_u=x_1x_2^2$. It is not difficult to show that $f_u$ is coercive  on $K$, i.e., $f_u(x^k)\to +\infty$ as $\| x^k\|\to +\infty$. Hence, the solution set of $\OP(K,f_u)$ is nonempty and bounded. %This yields $u\in\Rc(K,f)$.

            \item Second, we choose $u=(0,0)$, then $f_u=x_2(x_1x_2-1)$. This function is non-negative on $K$ and vanishes on $S:=\{(x_1,x_2)\in \R^2: x_1x_2=1,  x_1 \geq 1\}$. Hence, $S$ is the solution set of $\OP(K,f_u)$ that is unbounded.
        \end{itemize}

    \end{example}

    %The set is called the \textit{domain} of the weakly homogeneous optimization problem $\OP(K,f)$.

Suppose that $f$ is differentiable on an open set $U$ containing  $K$, we denote
$$\D(K,f):= (K^{\infty})^*-\nabla f(K).$$

\begin{remark}\label{rmk:DKf}
It is not difficult to see that $\K(K,f)=\K(K,f_u)$, for any $u\in \R^n$.
Furthermore, if $\K(K,f_u)$ is nonempty, then $\D(K,f)\subset \K(K,f_u)^*-\nabla f(K)$ because of $\K(K,f_u)\subset K^{\infty}$.
\end{remark}

The next theorem provides another sufficient condition concerning $\D(K,f)$ for the nonemptiness and boundedness of solutions to linearly parametric optimization problems.

\begin{theorem}\label{thm:4} Let $K$ be a convex set. Assume that $f$ is differentiable, convex, and bounded from below on $K$. If $u\in\inte(\D(K,f))\neq \emptyset$, then $\Sol(K,f_u)$ is nonempty and bounded.
    \end{theorem}
    \begin{proof} Because $f$ is bounded from below on $K$, from Proposition \ref{prop:nonempty}, $\K(K,f)$ is nonempty.
Suppose that $u$ belongs to $\inte(\D(K,f))\neq \emptyset$. It follows from Remark \ref{rmk:DKf} that
$u\in \inte(\K(K,f_u)^*-\nabla f(K))$.
Applying Lemma \ref{lm:conc} to the case $V=\K(K,f_u)$, we conclude that, for each $v\in \K(K,f_u)\setminus\{0\}$, there is $x\in K$ such that $\left\langle \nabla f(x)+u,v \right\rangle> 0$. Note that $f+\left\langle u,\cdot \right\rangle$ is convex on $K$.

Because $\K(K,f)$ is nonempty, we have the following two cases:
If $\K(K,f)$ is the trivial cone, then the nonemptiness and boundedness of $\Sol(K,f_u)$ follows Theorem \ref{thm:trivial}. If $\K(K,f)$ is non-trivial then, by applying Theorem \ref{pseudo} to the convex problem $\OP(K,f_u)$, one asserts that  $\Sol(K,f_u)$ is nonempty and bounded.  \qed
    \end{proof}

For  unconstrained optimization problems, we have the following property:

    \begin{corollary}\label{cor:2} Assume that $f$ is differentiable, convex on $\R^n$, and $\inte(\nabla f(\R^n))$ is nonempty. If $u\in-\inte(\nabla f(\R^n))$, then $\Sol(\R^n,f_u)$ is nonempty and bounded.
    \end{corollary}
    \begin{proof}The constraint set in the problem is $K=\R^n$. Because of $(K^{\infty})^*=\{0\}$, one has $\D(\R^n,f)=-\nabla f(\R^n)$. The conclusion in the corollary is straightforward from Theorem \ref{thm:4}.
    \end{proof}

    Consider a quadratic optimization problem $\OP(A,b,D,c)$ given by
    $f(x):=\frac{1}{2}x^TDx+c^Tx$, where $D\in \R^{n\times n}$ is symmetric, $c\in \R^n$, and  $K :=\{x\in\R^n:Ax\leq b\}$, where $A\in \R^{m\times n}$, $b\in \R^m$. We denote by $\Sol(A,b,D,c)$ the solution set of the problem. Assume that $K$ is unbounded and $h(x):=\frac{1}{2}x^TDx$ is an asymptotically homogeneous function of degree $\alpha=2$ of $f$ relative to $K$. Clearly, one has $\nabla f(x)=Dx$ and $K^{\infty} =\{x\in\R^n:Ax\leq 0\}$.

The following corollary follows directly from Theorem \ref{thm:4}.

    \begin{corollary}\label{cor:3} Asumme that $D$ is positive semidefinite on $K$. % and $\inte(A^T(\R^m_+)-D(K))$ is nonempty.
Then, the  set $\Sol(A,b,D,c+u)$ is nonempty and bounded if $u\in\inte((K^{\infty})^*-D(K))\neq\emptyset$.
    \end{corollary}

    To illustrate Corollary \ref{cor:3}, we consider the following example.

    \begin{example}\label{ex:6}
        Consider the convex quadratic optimization problem given by  $f(x_1,x_2):=(x_1-2x_2+1)^2$ and $K:=\{(x_1,x_2)\in \R^2: 0\leq x_1\leq 1\}$. We have $(K^{\infty})^*=\{(0,\mu): \mu\in \R\}$ and $\nabla f(K)= \{\gamma(1,-2): \gamma\in \R\}$. Besides, $\D(K,f) =\{(-\gamma,\mu+2\gamma): \mu, \gamma\in \R\}=\R^2$. From Theorem \ref{thm:4},  $\Sol(K,f_u)$ is nonempty and bounded, for any $u$ in $\R^2$.
%Furthermore, for  $u:=(\mu-\gamma,2\gamma)$, we see that
%$$\Sol(K,f_u)= \begin{cases}
%\{(x_1,x_2)\in \R^2: x_1-2x_2=\gamma-1,0\leq x_1\leq 1\} & \mbox{ if } \ \mu = 0,\\
%\{(0,\frac{1-\gamma}{2})\} &\mbox { if } \ \mu > 0, \\
%\{(1,\frac{2-\gamma}{2})\} &\mbox { if } \ \mu < 0. \\
%\end{cases}$$
    \end{example}

    \section{Upper Semicontinuity of the Solution Map}\label{sec:usc}

    This section is devoted to investigate the upper semicontinuity of the solution map $$\Sa:\R^{n}\rightrightarrows \R^n, \ \Sa(u)=\Sol(K,f_u).$$
    The two sets $\K(K,f)$ and $\D(K,f)$ play vital roles in this investigation.

    Recall that a set-valued map $\Phi:\R^m\rightrightarrows\R^n$ is \textit{locally bounded} at $\bar x$ if there exists an open neighborhood $U$ of $\bar x$ such that $\cup_{x\in U} \Phi(x)$ is bounded. %\cite[Definition 5.14]{RW}.
    The  map $\Phi$ is  \textit{upper semicontinuous} at $x\in T$, where $T$ is a subset of $\R^n$, if and only if for any open set $V\subset \R^n$ such that $\Phi(x)\subset V$ there exists a neighborhood $U$ of $x$ such that $\Phi(x')\subset V$ for all $x'\in U$. If $\Phi$ is upper semicontinuous at every $x\in T\subset \R^m$, then $\Phi$ is said that to be upper semicontinuous on $T$.

    \begin{remark}\label{local}
        If the set-valued map $\Phi$ is closed, i.e., its graph
        $$\gph(\Phi):=\big\{(u,v)\in \R^m\times \R^n: v\in \Phi(u)\big\}$$
        is closed in $\R^m\times \R^n$, and locally bounded at $x$, then $\Phi$ is upper semicontinuous at $x$ \cite[Theorem 5.19]{RW}.
    \end{remark}

    \begin{remark}\label{graph} The solution map $\Sa$ is closed. Indeed, we take a sequence $\{(u^k,x^k)\}$ in $\gph(\Sa)$ with
        $(u^k,x^k)\to (\bar u,\bar x).$
        It follows that $u^k\to\bar  u$ and $x^k\to \bar x$. Let $y\in K$ be arbitrary. By definition, one has
        $f(y)+ \left\langle u^k,y\right\rangle \geq f(x^k) +\left\langle u^k,x^k\right\rangle$ for any $k$.
        Taking $k\to+\infty$, by the continuity of $f$, we get $f(y)+ \left\langle \bar u,y\right\rangle \geq f(\bar x) +\left\langle \bar u,\bar x\right\rangle$. Thus, $\bar x\in \Sol(K,f_{\bar u})$; hence $\Sa$ is closed.
    \end{remark}

    \begin{theorem}\label{usc_int}
        Assume that $f$ is bounded from below on $K$. Then, $\Sa$ is upper
        semicontinuous on $\inte(\K(K,f)^*)$.
    \end{theorem}
    \begin{proof} From Proposition \ref{prop:nonempty}, $\K(K,f)$ is nonempty.
        By Remarks \ref{local} and \ref{graph}, we need only to prove that $\Sa$ is locally bounded at $u$, for $u\in \inte(\K(K,f)^*)$.
        Let $\varepsilon > 0$ be given, $\Bo(u,\varepsilon)$ and $\overline\Bo(u,\varepsilon)$ be the open ball and the closed ball, respectively, centered at $u$ of radius $\varepsilon$. Consider the following sets:
        \begin{equation}\label{M_eps}
            L:=\bigcup_{u\in \Bo( u,\varepsilon)} \Sa(u)\subset \bigcup_{u\in \overline\Bo(u,\varepsilon)} \Sa(u)=:R.
        \end{equation}
        We will show that $R$ is bounded.
        Suppose on the contrary that $R$ is unbounded. Then, there exist an unbounded sequence $\{x^k\}$ and a bounded sequence $\{u^k\}\subset\overline\Bo( u,\varepsilon)$ such that $x^k$ solves $\OP(K,f_{u^k})$ with $x^k\neq 0$ for every $k$, $\|x^k\|\to+\infty$, and
        $\|x^k\|^{-1}x^k\to\bar x$ with $\|\bar x\|=1$.
        By the compactness of $\overline\Bo(u,\varepsilon)$,  we can assume that $u^k\to \bar u$ with $\bar u\in\overline\Bo(u,\varepsilon)$.

        By assumptions, for every $k$, one has
        \begin{equation}\label{VI_k}
            f(y)+\left\langle u^k,y\right\rangle \geq f(x^k)+\left\langle u^k,x^k\right\rangle
        \end{equation}
        for all $y\in K$. Fix $\bar y \in K$. Then
        \eqref{VI_k} yields $f(\bar y)+\left\langle u^k,\bar y\right\rangle \geq f(x^k)+\left\langle u^k,x^k\right\rangle$. Dividing this inequality by $\|x^k\|^{\alpha}$ and letting $k\to+\infty$, we get $0 \geq f^{\infty}(\bar x)$.
It follows from the nonemptiness of $\K(K,f)$ and Remark \ref{rmk2} that $\bar x \in \K(K,f)$. Furthermore, since $f$ is bounded from below on $K$ by $\gamma$, from
        \eqref{VI_k} we have
        $$	f(\bar y)+\left\langle u^k,\bar y\right\rangle -\gamma \geq \left\langle u^k,x^k\right\rangle.$$
        Dividing this inequality by $\|x^k\|$ and letting $k\to+\infty$, we get $\left\langle \bar u,\bar x\right\rangle \leq 0$.
        Since $\bar u\in \inte(\K(K,f)^*)$ and $\bar x \in \K(K,f)$, we have $\left\langle \bar u,\bar x\right\rangle > 0$. Hence, we have an contradiction.
        \qed
    \end{proof}

    When the kernel is trivial, we get the upcoming corollary of Theorem \ref{usc_int}

    \begin{corollary}\label{usc_1} If $\K(K,f)=\{0\}$, then $\Sa$ is upper semicontinuous on $\R^n$.
    \end{corollary}
    \begin{proof}
        The proof is straightforward from Theorem \ref{usc_int} with the note that the dual cone of $\{0\}$ is the whole space $\R^n$. \qed
    \end{proof}

    \begin{example}The solution map of the optimization problem in Example \ref{ex:4} is upper
        semicontinuous on the open half-plane $\{(u_1,u_2)\in\R^2: u_1 > 0\}$ by Theorem \ref{usc_int}.
    \end{example}

    \begin{theorem}\label{thm:usc_Kinfty}
        Assume that $K$ is convex, $f$ is differentiable, convex, and bounded from below on $K$. Then, $\Sa$ is upper semicontinuous on $\inte(\D(K,f))$.
    \end{theorem}
    \begin{proof} From Proposition \ref{prop:nonempty}, $\K(K,f)$ is nonempty. Let $u\in \inte(\D(K,f))$ and take $\varepsilon>0$ such that $\overline\Bo(u,\varepsilon)\subset \inte(\D(K,f))$. We  need to prove that $\Sa$ is locally bounded at $u$. By repeating the argument and the notations from the proof of Theorem \ref{usc_1} with two sequences $\{x^k\}$, $\{u^k\}$ such that $x^k$ solves $\OP(K,f_{u^k})$, $\{x^k\}$ is unbounded, with $\|x^k\|^{-1}x^k\to\bar x$ and $u^k\to \bar u$, we can prove that $\bar x\in\K(K,f)$. Clearly, $f_{u^k}$ is convex on $K$. By Remark~\ref{rmk:Minty}, one has
        $\left\langle \nabla f(x)+u^k,x-x^k \right\rangle \geq 0,
        $ for all $x\in K$.	We fix $x\in K$, divide the last inequality by $\|x^k\|$, and let $k\to+\infty$; then we get $\left\langle \nabla f(x) + \bar u,\bar x \right\rangle \leq 0$. From Lemma~\ref{lm:conc}, we conclude that
        $u\notin\inte(\D(K,f))$. This contradicts our assumption. \qed
    \end{proof}

    The following example illustrates the previous theorem.

    \begin{example}
        Consider the convex optimization problem given as in Example~\ref{ex:6}. As $\D(K,f) = \R^2 = \Ri(K,f)$, from Theorem \ref{thm:usc_Kinfty}, $\Sa$ is upper semicontinuous on $\R^2$.
    \end{example}

    \section*{Acknowledgments} The author would like to express his deep gratitude to the two referees for their corrections and valuable comments which helped to improve the manuscript. This work has been supported by European Union’s Horizon 2020 research and innovation programme under the Marie Skłodowska-Curie Actions, grant agreement 813211 (POEMA).

    %=============================================References=================================================%

\end{document}